\documentclass[11pt]{amsart}
\usepackage{rotating}
\usepackage{amsthm}
 \usepackage{amsmath}
\usepackage{graphics}
 \usepackage{amsfonts}
 \usepackage{amssymb}
 \usepackage{amscd}
 \usepackage[all]{xy}
\usepackage{colordvi}
\usepackage{color}
\usepackage{pdfsync}
\usepackage{url}
\usepackage{enumerate}
\usepackage{booktabs}

\newtheorem{thm}{Theorem}[section]

\theoremstyle{definition}
\newtheorem{ex}[thm]{Example}

\newtheorem{alg}[thm]{Algorithm}
\theoremstyle{definition}

\newtheorem{rem}[thm]{Remark}
\numberwithin{equation}{section}
\newtheorem{summ}[thm]{Summary}

\def \N { {\mathbb N} }

\def \E { {\mathbb E} }

\def\ve#1{\mathchoice{\mbox{\boldmath$\displaystyle\bf#1$}}
{\mbox{\boldmath$\textstyle\bf#1$}}
{\mbox{\boldmath$\scriptstyle\bf#1$}}
{\mbox{\boldmath$\scriptscriptstyle\bf#1$}}}

\title{Estimating the number of zero-one multi-way tables via sequential importance sampling}
\author{Jing Xi \and Ruriko Yoshida \and David Haws}
\date{}

\begin{document}

\maketitle                 

\begin{abstract}
In 2005, Chen et al introduced a sequential importance sampling (SIS) procedure
to analyze  zero-one two-way tables with given fixed marginal sums (row and
column sums) via the conditional Poisson (CP)
distribution.  They showed that compared with Monte Carlo Markov chain 
(MCMC)-based approaches, their importance sampling method is
more efficient in terms of running time and also provides an easy and
accurate estimate of the total number of contingency tables with fixed marginal 
sums.  In this paper we extend their result to zero-one
multi-way ($d$-way, $d \geq 2$) contingency tables under the no
$d$-way interaction model, i.e., with fixed $d - 1$ marginal sums.  
Also we show by simulations that the SIS procedure with CP distribution to
estimate the number of zero-one
three-way tables under the no three-way interaction model given marginal
sums works very well even with some rejections.  We also applied our
method to Samson's monks' data set.  We end with
further questions on the SIS procedure on zero-one multi-way tables.
\end{abstract}

\section{Introduction}

Sampling zero-one constrained contingency tables finds its applications
in combinatorics \cite{Huber06}, statistics of social networks
\cite{chen2007, Snijders1991}, 
and regulatory networks \cite{Dinwoodie2008}.  
In 2005, Chen et al.\ introduced a sequential importance sampling (SIS) procedure
to analyze  zero-one two-way tables with given fixed marginal sums (row and
column sums) via the conditional Poisson (CP)
distribution \cite{chen2005}.  
It proceeds by simply sampling cell entries of the zero-one contingency
table sequentially for each row 
such that the final distribution approximates the target distribution. This
method will terminate at the last column and sample independently and
identically distributed (iid) tables from the
proposal distribution. Thus the SIS procedure does not require expensive or
prohibitive pre-computations, as is the case of computing Markov
bases for the Monte Carlo Markov Chain (MCMC)
approach. Also, when attempting to sample a single table, 
if there is no rejection, the SIS procedure is guaranteed to sample a
table from the distribution, where in an
MCMC approach the chain may require a long time to run in order to
satisfy the independent condition.

In 2007, Chen extended their SIS procedure to
sample zero-one two-way tables with given fixed row and
column sums with structural zeros, i.e., some cells are
constrained to be zero or one \cite{chen2007}.  
In this paper we also extended the results from
\cite{chen2005,chen2007} to zero-one
multi-way ($d$-way, $d \geq 2$) contingency tables under the no
$d$-way interaction model, i.e., with fixed $d - 1$ marginal sums.  

This paper is organized as follows:  In Section \ref{sec:sis} we
outline basics of the SIS procedure.   In Section \ref{3dim} we focus on
the SIS procedure with CP distribution on three-way tables under no
three-way interaction model.  This model is particularly important
since if we are able to count or estimate the number of tables under this
model then this is equivalent to estimating the number of {\em lattice
  points} in any {\em polytope} \cite{deloera06}.  This means that if
we can estimate the number of three-way zero-one tables under this model,
then we can estimate the number of any zero-one tables by using De Loera
and Onn's bijection mapping.     

Let $\ve X=(X_{i j k})$ of size
$(m \mbox{, }n \mbox{, } l)$, where $m, n, l \in \N$ and $\N = \{1, 2,
\ldots, \}$, be a table of counts whose entries are 
independent Poisson random variables with canonical parameters
$\{\theta_{ijk}\}$. Here $X_{ijk} \in \{0, 1\}$.  Consider the generalized linear model,
\begin{eqnarray}
\label{eq:RCmodel}
  \theta_{ijk} = \lambda + \lambda^M_i + \lambda^N_j + \lambda^L_k +
  \lambda_{ij}^{MN} + \lambda_{ik}^{ML} +\lambda_{jk}^{NL}\,
\end{eqnarray}
for $i=1,\ldots,m$, $j=1,\ldots,n$, and $k = 1, \ldots , l$ where $M$,
$N$, and $L$ denote the
nominal-scale factors.  This model is called the {\em no three-way
  interaction model}.

Notice that the sufficient statistics under the
model in \eqref{eq:RCmodel} are the {\em two-way marginals}, that is:
\begin{equation}\label{tableEqu}
 \begin{array}{ll}
X_{+jk} := \sum _{i=1}^m X_{i j k} \mbox{, } (j=1,2,\ldots,n\mbox{, } 
k=1,2,\ldots,l),\\
X_{i+k} := \sum _{j=1}^n X_{i j k}\mbox{, } (i=1,2,\ldots,m\mbox{, } 
k=1,2,\ldots,l),\\
X_{ij+} := \sum _{k=1}^l X_{i j k}\mbox{, } (i=1,2,\ldots,m\mbox{, } 
j=1,2,\ldots,n),\\
\end{array}
\end{equation}
Hence, the
conditional distribution of the table counts given the margins is the
same regardless of the values of the parameters in the model.

In Section \ref{4dim} we generalize the SIS procedure on zero-one
two-way tables in \cite{chen2005,chen2007} to zero-one
multi-way ($d$-way, $d \geq 2$) contingency tables under the no
$d$-way interaction model, i.e., with fixed $d - 1$ marginal sums.
In Sections \ref{comp} and \ref{sam} we show some simulation results with our
software which is available in \url{http://www.polytopes.net/code/CP}.  Finally, we end with some discussions.

\section{Sequential importance sampling}\label{sec:sis}

Let $\Sigma$ be the set of all tables satisfying marginal
conditions.  In this paper we assume that $\Sigma \not = \emptyset$. 
Let $P({\bf X})$ for any ${\bf X}
\in \Sigma$ be the uniform distribution over $\Sigma$, so $p({\bf X}) =
1/|\Sigma|$.  Let $q(\cdot)$ be a trial distribution such that $q({\bf X}) >
0$ for all ${\bf X} \in \Sigma$.  Then we have 
\[
\E \left[\frac{1}{q({\bf X})}\right] = \sum_{{\bf X} \in \Sigma} \frac{1}{q({\bf X})} q({\bf X}) = |\Sigma|.
\]
Thus we can estimate $|\Sigma |$ by
\[
\widehat{|\Sigma|} = \frac{1}{N} \sum_{i = 1}^N \frac{1}{q({\bf X_i})},
\]
where ${\bf X_1}, \ldots , {\bf X_N}$ are tables drawn iid from $q({\bf X})$.
Here, this proposed distribution $q({\bf X})$ is the distribution
(approximate) to sample tables via the SIS procedure.  

We vectorized the table ${\bf X} = (x_1, \cdots , x_t)$ and by the
multiplication rule we have
\[
q({\bf X} = (x_1, \cdots , x_t)) = q(x_1)q(x_2|x_1)q(x_3|x_2, x_1)\cdots
q(x_t|x_{t-1}, \ldots , x_1). 
\]
Since we sample each cell count of a table from an interval we can
easily compute $q(x_i|x_{i-1}, \ldots , x_1)$ for $i = 2, 3, \ldots ,
t$.

When we have rejections, this means that we are sampling tables from a
bigger set $\Sigma^*$ such that $\Sigma \subset \Sigma^*$. In this
case, as long as the conditional probability $q(x_i|x_{i-1}, \ldots ,
x_1)$ for $i = 2, 3, \ldots$ and $q(x_1)$ are normalized, $q({\bf X})$
is normalized over $\Sigma^*$ since
\[
\begin{array}{rcl}
\sum_{{\bf X} \in \Sigma^*} q({\bf X}) &=& \sum_{x_1, \ldots, x_t} q(x_1) q(x_2|x_1)q(x_3|x_2, x_1)\cdots
q(x_t|x_{t-1}, \ldots , x_1)\\ 
&= & \sum_{x_1} q(x_1) \left[ \sum_{x_2}
  q(x_1|x_2) \left[ \cdots
    \left[ \sum_{x_t} q(x_t|x_{t-1}, \ldots , x_1)\right] 
  \right]\right]\\
& = & 1.\\
\end{array}
\]
Thus we have 
\[
\E \left[\frac{\mathbb{I}_{{\bf X} \in \Sigma}}{q({\bf X})}\right] = \sum_{{\bf
    X} \in \Sigma^*} \frac{\mathbb{I}_{{\bf X} \in \Sigma}}{q({\bf X})}
q({\bf X}) = |\Sigma|,
\]
where 
$\mathbb{I}_{{\bf X} \in \Sigma}$ is an indicator function for the set
$\Sigma$.    By the law of large numbers this estimator is unbiased.  

\section{Sampling from the conditional Poisson distribution}\label{3dim}

Let 
\[
Z = (Z_1, \ldots ,Z_l) 
\]
be independent Bernoulli trials with probability of successes
$p = ( p_1, \ldots , p_l)$. Then the random variable
\[
S_Z = Z_1 +\cdots +Z_l
\]
is a Poisson--binomial distribution.

We say the column of entries for the marginal $X_{i_0, j_0, +}$ of $\ve
X$ is the
$(i_0, j_0)$th column of $\ve X$ (equivalently we say $(i_0, k_0)$th column for the 
marginal $X_{i_0 + k_0}$ and $(j_0, k_0)$th column for the marginal $X_{+j_0k_0}$).
Consider the $(i_0, j_0)$th
column of the table $\ve X$ for some $i_0 \in \{1, \ldots, m\}$, $j_0
\in \{1, \ldots , n\}$ with the marginal $l_0 = X_{i_0j_0+}$.  Also we
let $r_k = X_{i_0+k}$ and  $c_k = X_{+j_o k}$.
Now let $w_k = p_k/(1 - p_k)$ where $p_k \in (0, \, 1)$.
Then,
\begin{equation}\label{dist}
P(Z_1 = z_1, \ldots ,Z_l = z_l|S_Z = l_0) \propto \prod_{k = 1}^l w_k^{z_k}.
\end{equation}

Thus for sampling a zero-one table with fixed marginals $X_{+jk}, \,
X_{i+k}$ for $i=1,2,\ldots ,m\mbox{, }  
j=1,2,\ldots ,n$, and $k = 1, 2, \ldots, l$, for
$X_{i_0j_0+}$ for each $i_0 \in \{1, \ldots, m\}$ and $j_0 \in \{1, 
\ldots , n\}$, (or one can do each $X_{i_0+k_0}$ or $X_{+j_0k_0}$
instead by similar way) one just decides which entries are ones
(basically there are ${l \choose l_0}$ many choices) using
the conditional Poisson distribution above.  We sample these
cell entries with ones (say $l_0$ many entries with ones) in the $(i_0,
j_0)$th column for the $L$ factor with the
following probability:
Let $A_k$, for $k = 1, \ldots , l_0$, be the set of selected entries.
Thus $A_0 = \emptyset$, and $A_{l_0}$ is the final sample that we obtain. At the
$k$th step of the drafting sampling $(k = 1, \ldots , l_0)$, a unit $j \in A^c_{k-1}$
is selected into the sample with probability
\[
P(j, A^c_{k-1}) = \frac{w_j R(l_0 - k, A^c_{k-1} - j)}{(l_0 - k +
  1)R(l_0-k+1, A^c_{k-1})},
\]
where 
\[
R(s, A) = \sum_{B \subset A, |B| = s} \left(\prod_{i \in B}w_i \right).
\]

For sampling a zero-one three-way table $\ve X$ with given two-way marginals
$X_{ij+}$, $X_{i+k}$, and $X_{+jk}$ for $i=1,2,\ldots ,m\mbox{, }  
j=1,2,\ldots ,n$, and $k = 1, 2, \ldots, l$, we sample for the $(i_0, j_0)$th
column of the table $\ve X$ for each $i_0 \in \{1, \ldots , m\}$, $j_0
\in \{1, \ldots , n\}$.  We set 
\begin{equation}\label{pr}
p_k := \frac{r_k \cdot c_k}{r_k \cdot c_k + (n - r_k)(m - c_k)}.
\end{equation}
Thus we have 
\begin{equation}\label{wt}
w_k = \frac{r_k \cdot c_k}{(n - r_k)(m - c_k)}.
\end{equation}

\begin{rem}
We assume that we do not have the trivial cases, namely,  $1 \leq r_k
\leq n-1$ and $1 \leq c_k \leq m - 1$. 
\end{rem}

\begin{thm}\label{main}
For the uniform distribution over all $m\times n \times l$ zero-one tables with given marginals  $r_k = X_{i_0+k}, \, c_k = X_{+j_0k}$ for
$k = 1, 2, \ldots, l$, and a fixed marginal for
the factor $L$, $l_0$, the marginal distribution of the fixed marginal
$l_0$ is the
same as the conditional distribution of $Z$ defined by \eqref{dist} given
$S_Z = l_0$ with 
\[
p_k :=  \frac{r_k \cdot c_k}{r_k \cdot c_k + (n - r_k)(m - c_k)}.
\]

\end{thm}

\begin{proof}
We start by giving an algorithm for generating tables uniformly
from all $m \times n \times l$ zero-one tables with given
marginals  $r_k, \, c_k$ for 
$k = 1, 2, \ldots, l$, and a fixed marginal for
the factor $L$, $l_0$.

\begin{enumerate}
\item\label{step1} For $k = 1, \ldots , l$ consider the $k$th layer of $m \times n$
  tables.  We randomly choose $r_k$ positions in the $(i_0, k)$th column
  and $c_k$ positions in the $(j_0, k)$th column,
and put $1$’s in those positions. The choices of positions are independent across
different layers.
\item\label{step2} Accept those tables with given column sum $l_0$.
\end{enumerate}
It is easy to see that tables generated by this algorithm are uniformly
distributed over all $m\times n \times l$ zero-one tables with given
marginals $r_k, \, c_k$ for 
$k = 1, 2, \ldots, l$, and a fixed marginal for
the factor $L$, $l_0$ for the $(i_0, j_0)$th column of the table $\ve X$.  We can derive the marginal distribution
of the $(i_0, j_0)$th column of $\ve X$ based on this algorithm. At Step \ref{step1}, we choose
the cell at position $(i_0, \, j_0, \,  1)$ to put $1$ in
with the probability:
\[
\frac{{n - 1\choose r_1 - 1}{m - 1 \choose c_1 - 1}}{{n - 1 
    \choose r_1 - 1}{m - 1 \choose c_1 - 1} + {n - 1 \choose r_1}{m -
    1 \choose c_1}} =  \frac{r_1 \cdot c_1}{r_1 \cdot c_1 + (n - r_1)(m - c_1)}.
\]
Because the choices of positions are independent across different layers,
after Step \ref{step1} the marginal distribution of the $(i_0, j_0)$th column is the same as
the distribution of $Z$ defined by \eqref{dist} with
\[
p_k = \frac{{n - 1\choose r_k - 1}{m - 1 \choose c_k - 1}}{{n - 1 
    \choose r_k - 1}{m - 1 \choose c_k - 1} + {n - 1 \choose r_k}{m -
    1 \choose c_k}} =  \frac{r_k \cdot c_k}{r_k \cdot c_k + (n - r_k)(m - c_k)}.
\]
 Step \ref{step2} rejects the
tables whose $(i_0, j_0)$th column sum is not $l_0$. This implies that after Step \ref{step2},
the marginal distribution of the $(i_0, j_0)$th column is the same as the conditional
distribution of $Z$ defined by \eqref{dist} with
\[
p_k =\frac{r_k \cdot c_k}{r_k \cdot c_k + (n - r_k)(m - c_k)}.
\]
\end{proof}

\begin{rem}
The sequential importance sampling via CP for sampling a two-way zero-one
table  defined in \cite{chen2005} is a special case of our SIS procedure.  
We can induce $p_k$ defined in \eqref{pr} and  the weights defined in
\eqref{wt} to the weights for two-way zero-one contingency tables
defined in \cite{chen2005}.  Note that when we 
consider two-way zero-one contingency tables we have 
$c_k = 1 $ for all $k = 1, \ldots, l$ and for all $j_0 = 1, \ldots, n$ (or 
$r_k = 1 $ for all $ k = 1, \ldots, l$ and for all $i_0 = 1, \ldots,
m$), and $m = 2$ (or $n = 2$, respectively).  
Therefore when we consider the two-way zero-one tables we get
\[
p_k = \frac{r_k}{n}, \, w_k = \frac{r_k}{n - r_k},
\]
or respectively
\[
p_k = \frac{c_k}{m}, \, w_k = \frac{c_k}{m - c_k}.
\]
\end{rem}

During the intermediary steps of our SIS procedure via CP on a three-way zero-one table 
there will be
some columns for the $L$ factor with trivial cases.  In that case we
have to treat them as structural zeros in the $k$th slice for some $k
\in \{1, \ldots , l\}$.  In that case we have to use the
probabilities for the distribution in \eqref{dist} as follows:
\begin{equation}\label{pr3}
p_k := \frac{r_k \cdot c_k}{r_k \cdot c_k + (n - r_k - g_k^{r_0})(m - c_k
  - g_k^{c_0})},
\end{equation}
where $g_k^{r_0}$ is the number of structural zeros in the $(r_0, k)$th
column and $g_k^{c_0}$ is the number of structural zeros in the $(c_0,
k)$th
column.
Thus we have weights:
\begin{equation}\label{wt3}
w_k = \frac{r_k \cdot c_k}{(n - r_k - g_k^{r_0})(m - c_k - g_k^{c_0})}.
\end{equation}

\begin{thm}\label{main_str1}
For the uniform distribution over all $m\times n \times l$ 
zero-one tables with structural zeros with given marginals  $r_k = X_{i_0+k}, \, c_k = X_{+j_0k}$ for
$k = 1, 2, \ldots, l$, and a fixed marginal for
the factor $L$, $l_0$, the marginal distribution of the fixed marginal
$l_0$ is the
same as the conditional distribution of $Z$ defined by \eqref{dist} given
$S_Z = l_0$ with 
\[
p_k := \frac{r_k \cdot c_k}{r_k \cdot c_k + (n - r_k - g_k^{r_0})(m - c_k
  - g_k^{c_0})},
\]
where $g_k^{r_0}$ is the number of structural zeros in the $(r_0, k)$th
column and $g_k^{c_0}$ is the number of structural zeros in the $(c_0,
k)$th column.
\end{thm}
\begin{proof}
The proof is similar to the proof for Theorem \ref{main}, just replace the
probability $p_k$ with
\[
p_k = \frac{{n - 1  - g_k^{r_0}\choose r_k - 1}{m - 1  - g_k^{c_0}\choose c_k - 1}}{{n - 1 
    - g_k^{r_0} \choose r_k - 1 }{m - 1  - g_k^{c_0} \choose c_k - 1 } + {n - 1  - g_k^{r_0}\choose r_k}{m -
    1  - g_k^{c_0}\choose c_k}} =  \frac{r_k \cdot c_k}{r_k \cdot c_k
  + (n - r_k  - g_k^{r_0})(m - c_k  - g_k^{c_0})}.
\]
\end{proof}

\begin{rem}
The sequential importance sampling via CP for sampling a two-way zero-one
table with structural zeros  defined in Theorem 1 in \cite{chen2007} is a special case of our SIS.  
We can induce $p_k$ defined in \eqref{pr3} and  the weights defined in
\eqref{wt3} to the weights for two-way zero-one contingency tables
defined in \cite{chen2007}.  Note that when we 
consider two-way zero-one contingency tables we have 
$c_k = 1 $ for all $k = 1, \ldots, l$ and for all $j_0 = 1, \ldots, n$ (or 
$r_k = 1 $ for all $ k = 1, \ldots, l$ and for all $i_0 = 1, \ldots,
m$),  $m = 2$ (or $n = 2$, respectively), and $g_k^{c_0} = 0$ (or
$g_k^{r_0}$, respectively).  
Therefore when we consider the two-way zero-one tables we get
\[
p_k = \frac{r_k}{n  - g_k^{r_0}}, \, w_k = \frac{r_k}{n - r_k - g_k^{r_0}},
\]
or respectively
\[
p_k = \frac{c_k}{m  - g_k^{c_0}}, \, w_k = \frac{c_k}{m - c_k - g_k^{c_0}}.
\]
\end{rem}

\begin{figure}[!htp]
\begin{center}
\scalebox{0.6}{
\includegraphics{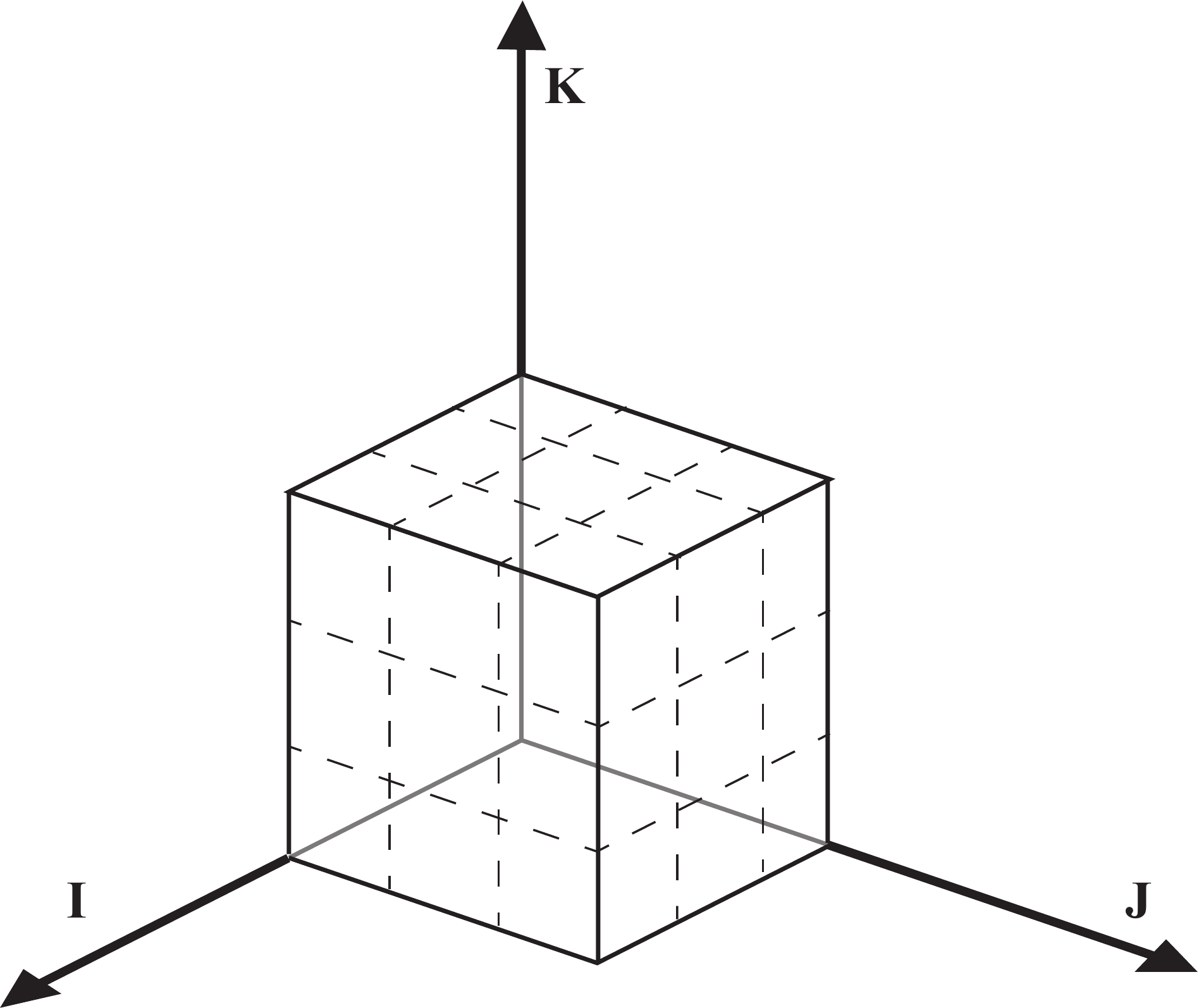}
}
\end{center}
\caption{An example of a $3\times 3\times 3$ table.}
\label{cube}
\end{figure}


\begin{alg}[Store structures in the zero-one table]\label{alg1}

This algorithm stores the structures, including zeros and ones, in
the observed table $\ve x_0$. The output will be used to avoid trivial
cases in sampling. The output $A$ and $B$ matrices both have the same
dimension with $\ve x_0$, so the cell value in $A$ will be $1$ if the
position is structured and $0$ if not. The matrix $B$ is only for structure $1$'s. We
consider sampling a table without structure $1$'s, that is,
a table with new marginals: $X^*_{ij+}=X_{ij+}-\sum_{k=1}^l
B_{ijk}=X_{ij+}-B_{ij+}$, $X^*_{i+k}=X_{i+k}-\sum_{j=1}^n
B_{ijk}=X_{i+k}-B_{i+k}$, and $X^*_{+jk}=X_{+jk}-\sum_{i=1}^m
B_{ijk}=X_{+jk}-B_{+jk}$ for $i=1,2,\ldots ,m\mbox{, }  j=1,2,\ldots
,n$, and $k = 1, 2, \ldots, l$. 

\begin{itemize}
\item[{\bf Input}] The observed marginals $X_{ij+}$, $X_{i+k}$, and $X_{+jk}$ for $i=1,2,\ldots ,m\mbox{, }  j=1,2,\ldots ,n$, and $k = 1, 2, \ldots, l$.
\item[{\bf Output}] Matrix $A$ and $B$, new marginals $X^*_{ij+}$, $X^*_{i+k}$, and $X^*_{+jk}$ for $i=1,2,\ldots ,m\mbox{, }  j=1,2,\ldots ,n$, and $k = 1, 2, \ldots, l$.
\item[{\bf Algorithm}]
\begin{enumerate}
\item
Check all marginals in direction I. For $i = 1, 2, \ldots, m$:\\
If $X_{+jk} = 0$, $A_{i'jk} = 1$, for all $i' = 1, 2, \ldots, m$ and $A_{i'jk}=0$;\\
If $X_{+jk} = 1$, $A_{i'jk} = 1$ and $B_{i'jk} = 1$, for all $i' = 1, 2, \ldots, m$ and $A_{i'jk}=0$.\\
\item
Check all marginals in direction J. For $j = 1, 2, \ldots, n$:\\
If $X_{i+k} = 0$, $A_{ij'k} = 1$, for all $j' = 1, 2, \ldots, n$ and $A_{ij'k}=0$;\\
If $X_{i+k} = 1$, $A_{ij'k} = 1$ and $B_{ij'k} = 1$, for all $j' = 1, 2, \ldots, n$ and $A_{ij'k}=0$.\\
\item
Check all marginals in direction K. For $k = 1, 2, \ldots, l$:\\
If $X_{ij+} = 0$, $A_{ijk'} = 1$, for all $k' = 1, 2, \ldots, l$ and $A_{ijk'}=0$;\\
If $X_{ij+} = 1$, $A_{ijk'} = 1$ and $B_{ijk'} = 1$, for all $k' = 1, 2, \ldots, l$ and $A_{ijk'}=0$.\\
\item 
If any changes made in step (1), (2) or (3), come back to (1), else stop.
\item
Compute new marginals: \\
$X^*_{ij+}=X_{ij+}-B_{ij+}$, $X^*_{i+k}=X_{i+k}-B_{i+k}$, and
$X^*_{+jk}=X_{+jk}-B_{+jk}$ for $i=1,2,\ldots ,m\mbox{, }
j=1,2,\ldots ,n$, and $k = 1, 2, \ldots, l$. 
\end{enumerate}
\end{itemize}
\end{alg}

\begin{alg}[Generate a two-way table with given marginals]\label{alg2}

This algorithm is used to generate a layer (fixed $i$) of the three-way
table, with the probability of the sampled layer. 

\begin{itemize}
\item[{\bf Input}] Row sums $r^*_j$ and column sums $c^*_k$, $ j=1,2,\ldots
  ,n$, and $k = 1, 2, \ldots, l$; structures $A$; marginals on direction
  I: $X_{+jk}$ for $i=1,2,\ldots ,m$. 
\item[{\bf Output}] A sampled table and its probability. Return $0$ if the process fails.
\item[{\bf Algorithm}]
\begin{enumerate}
\item
Order all columns with decreasing sums.
\item
Generate the column (along the direction $K$) with the largest sum,
the weights used in CP are shown in equation \eqref{wt3}. Notice that
$k$ relates to each specific cell in the column,  $r_k$ and $c_k$
which are 
the row sums in the direction $J$ and $I$, respectively.  $g^{r_0}_k$
and $g^{c_0}_k$ are the number of structures in the rows of the
direction $J$ and $I$, respectively. The probability of the generated
column will be returned if the process succeeds, while $0$ may be
returned in this step if it does not exist. 
\item
Delete the generated column in (2), and for the remaining subtable, do
the following:
\begin{enumerate}
\item
If only one column is left, fill it with fixed marginals and go to (4).
\item
If (a) is not true, check all marginals to see if there are any new structures caused
by step (2). We need to avoid trivial cases by doing this. Go back to
(1) with new marginals and structures.  
\end{enumerate}
\item
Return generated matrix as the new layer and its CP probability. If failed, return $0$.
\end{enumerate}
\end{itemize}
\end{alg}

\begin{alg}[SIS with CP for sampling a three-way zero-one table]\label{alg3}

We describe an algorithm to sample a three-way zero-one table $\ve X$ with given
marginals $X_{ij+}$, $X_{i+k}$, and $X_{+jk}$ for $i=1,2,\ldots ,m\mbox{, }  
j=1,2,\ldots ,n$, and $k = 1, 2, \ldots, l$ via the SIS with CP.

\begin{itemize}
\item[{\bf Input}] The observed table $\ve x_0$.
\item[{\bf Output}] The sampled table $\ve x$.
\item[{\bf Algorithm}]
\begin{enumerate}
\item
Compute the marginals $X_{ij+}$, $X_{i+k}$, and $X_{+jk}$ for $i=1,2,\ldots ,m\mbox{, }  
j=1,2,\ldots ,n$, and $k = 1, 2, \ldots, l$.
\item
Use Algorithm \ref{alg1} to compute the structure tables $A$ and 
$B$. Consider the new marginals in the output as the sampling
marginals. 
\item
For the sampling marginals, do the SIS:
\begin{enumerate}
\item
Delete the layers filled by structures; consider the left-over subtable.
\item
Consider the layers in direction $I$ ($i$ varies). Sum within all
layers and order them from the largest to smallest. 
\item
Consider the layer with the largest sum and plug in the structure table
$A$ from Algorithm \ref{alg2} to generate a sample for this
layer. The algorithm may return $0$ if the sampling fails. 
\item
Delete the generated layer in (c), and for the remaining subtable, do the following:
\begin{enumerate}
\item
If only one layer left, fill it with fixed marginals and go to (e).
\item
else, go back to (2) with new marginals.
\end{enumerate}
\item
Add the sampled table with table $B$ (the structure $1$'s table).
\end{enumerate}
\item
Return the table in (e) and the same probability with the sampled table. Return $0$ if failed.
\end{enumerate}
\end{itemize}
\end{alg}

\section{Four or higher dimensional zero-one tables}\label{4dim}

In this section we consider a $d$-way zero-one table under the no
$d$-way interaction model for $d \in \N$ and $d > 3$.  
Let $\ve X=(X_{i_1 \ldots i_d})$ be a zero-one contingency table of size
$(n_1 \times \cdots \times n_d)$, where $n_i \in \N$ for $i = 1,
\ldots, d$.  
The sufficient statistics under  the no
$d$-way interaction model are
\begin{equation}\label{margin}
\begin{array}{l}
X_{+i_2\ldots i_d}, \, X_{i_1+i_3\ldots i_d}, \, \ldots , X_{i_1\ldots
  i_{d-1}+},\\
\text{for } i_1 = 1, \ldots, n_1, \, i_2 = 1, \ldots, n_2, \ldots , i_d
= 1, \ldots, n_d.\\
\end{array}
\end{equation}

For each $i_1^0 \in \{1, \ldots, n_1\}, \ldots , i_{d-1}^0 \in \{1,
\ldots, n_{d}\}$, we say the column of the entries for a marginal
$X_{i_1 \ldots i_{j - 1}+ i_{j+1}\ldots i_d }$ the $(i_0, \ldots ,
i_{j-1}, i_{j+1}, \ldots , i_d)$th column of $\ve X$.
For each $i_1^0 \in \{1, \ldots, n_1\}, \ldots , i_{d-1}^0 \in \{1,
\ldots, n_{d-1}\}$, we consider the $(i_1^0, \ldots, i_{d-1}^0)$th
column for the $d$th factor.  Let $l_0 = X_{i_1^0, \ldots, i_{d-1}^0+}$.  
Let $r_k^j = X_{i_1^0 \ldots i_{j-1}^0+i_{j+1}^0\ldots i_{d-1}^0k}$ for fixed $k \in
\{1, \ldots , n_d\}$.
For sampling a zero-one $d$-way table $\ve X$, we set
\begin{equation}\label{pr2}
p_k := \frac{\prod_{j = 1}^{d-1} r_k^j}{\prod_{j = 1}^{d-1} r_k^j + \prod_{j
    = 1}^{d-1}(n_j - r_k^j)}.
\end{equation}

\begin{rem}
We assume that we do not have trivial cases, namely,  $1 \leq r_k^j
\leq n_j-1$ for $j = 1, \ldots , d$. 
\end{rem}

\begin{thm}
For the uniform distribution over all $d$-way zero-one contingency
tables $\ve X=(X_{i_1 \ldots i_d})$ of size
$(n_1 \times \cdots \times n_d)$, where $n_i \in \N$ for $i = 1,
\ldots, d$ with marginals $l_0 = X_{i_1^0, \ldots, i_{d-1}^0+}$,
and $r_k^j = X_{i_1^0 \ldots i_{j-1}^0+i_{j+1}^0\ldots i_{d-1}^0k}$ for $k \in
\{1, \ldots , n_d\}$, the marginal distribution of the fixed marginal
$l_0$ is the
same as the conditional distribution of $Z$ defined by \eqref{dist} given
$S_Z = l_0$ with 
\[
p_k :=  \frac{\prod_{j = 1}^{d-1} r_k^j}{\prod_{j = 1}^{d-1} r_k^j + \prod_{j
    = 1}^{d-1}(n_j - r_k^j)}.
\]

\end{thm}
\begin{proof}
The proof is similar to the proof for Theorem \ref{main}, we just extend
the same argument to a $d$-way zero-one table under the no $d$-way
interaction model  with the probability
\[
p_k = \frac{\prod_{j = 1}^{d-1}{n_j - 1\choose r_k^j - 1}}{\prod_{j = 1}^{d-1}{n_j - 1 
    \choose r_k^j - 1} + \prod_{j=1}^{d-1}{n_j - 1 \choose r_k^j}} =
\frac{\prod_{j = 1}^{d-1} r_k^j}{\prod_{j = 1}^{d-1} r_k^j + \prod_{j   = 1}^{d-1}(n_j - r_k^j)}.
\]
\end{proof}

During the intermediary steps of our SIS procedure via CP on a three-way zero-one table 
there will be
some columns for the $d$th factor with trivial cases.  In that case we
have to treat them as structural zeros in the $k$th slice for some $k
\in \{1, \ldots , l\}$.  In that case we have to use the
probabilities for the distribution in \eqref{dist} as follows:
\begin{equation}\label{pr4}
p_k := \frac{\prod_{j = 1}^{d-1} r_k^j }{\prod_{j = 1}^{d-1} r_k^j + \prod_{j
    = 1}^{d-1}(n_j - r_k^j - g_k^j)}.
\end{equation}
where $g_k^{j}$ is the number of structural zeros in  the $(i_1^0, \ldots ,
i_{j-1}^0, i_{j+1}^0, \ldots , i_{d-1}^0k)$th column of $\ve X$.
Thus we have weights:
\begin{equation}\label{wt4}
w_k = \frac{\prod_{j = 1}^{d-1}r_k^j}{\prod_{j
    = 1}^{d-1}(n_j - r_k^j - g_k^j)}.
\end{equation}

\begin{thm}\label{main_str2}
For the uniform distribution over all $d$-way zero-one contingency
tables $\ve X=(X_{i_1 \ldots i_d})$ of size
$(n_1 \times \cdots \times n_d)$, where $n_i \in \N$ for $i = 1,
\ldots, d$ with marginals $l_0 = X_{i_1^0, \ldots, i_{d-1}^0+}$,
and $r_k^j = X_{i_1^0 \ldots i_{j-1}^0+i_{j+1}^0\ldots i_{d-1}^0k}$ for $k \in
\{1, \ldots , n_d\}$, the marginal distribution of the fixed marginal
$l_0$ is the
same as the conditional distribution of $Z$ defined by \eqref{dist} given
$S_Z = l_0$ with 
\[
p_k :=  \frac{\prod_{j = 1}^{d-1} r_k^j}{\prod_{j = 1}^{d-1} r_k^j + \prod_{j
    = 1}^{d-1}(n_j - r_k^j - g_k^j)}
\]
where $g_k^{j}$ is the number of structural zeros in the $(i_1^0, \ldots ,
i_{j-1}^0, i_{j+1}^0, \ldots , i_{d-1}^0k)$th column of $\ve X$.
\end{thm}
\begin{proof}
The proof is similar to the proof for Theorem \ref{main_str1}, we
just extend
the same argument to a $d$-way zero-one table under the no $d$-way
interaction model  with the probability
\[
p_k = \frac{\prod_{j = 1}^{d-1}{n_j - 1 - g^j_k\choose r_k^j - 1}}{\prod_{j = 1}^{d-1}{n_j - 1  - g^j_k
    \choose r_k^j - 1} + \prod_{j=1}^{d-1}{n_j - 1  - g^j_k\choose r_k^j}} =
\frac{\prod_{j = 1}^{d-1} r_k^j}{\prod_{j = 1}^{d-1} r_k^j + \prod_{j
    = 1}^{d-1}(n_j - r_k^j  - g^j_k)}.
\]
\end{proof}

\section{Computational examples}\label{comp}

For our simulation study we used the software package {\tt R} \cite{Rproj}.
We count the {\em exact} numbers of tables
via the software {\tt LattE} \cite{latte-1.2} for small examples in
this section (Examples \eqref{firstex} to \eqref{lastex}). 


When the contingency tables are large and/or the models are complicated, it
is very difficult to obtain the exact number of tables.  Thus we need a good
measurement of accuracy in the estimated number of tables. 
In \cite{chen2005}, they used the coefficient of variation
($cv^2$):
\[
cv^2 = \frac{var_q\{p({\bf X})/q({\bf X})\}}{\E^2_q\{p({\bf X})/q({\bf X})\}}
\]
which is equal to $var_q\{1/q({\bf X})\}/\E^2_q\{1/q({\bf X})\}$ for
the problem of estimating the number of tables.  The value of $cv^2$ is simply
the chi-square distance between the 
two distributions $p'$ and $q$, which means the smaller it is, the closer the two
distributions are.
In \cite{chen2005} they estimated $cv^2$ by:
\[
cv^2 \approx \frac{\sum_{i = 1}^N \{1/q({\bf X_i}) - \left[\sum_{j =1}^N 1/q({\bf X_j}) \right] / N  \}^2 /(N - 1)}{\left\{\left[ \sum_{j = 1}^N 1/q({\bf X_j})\right] / N\right\}^2},
\]
where ${\bf X_1}, \ldots , {\bf X_N}$ are tables drawn iid from
$q({\bf X})$.   When we have rejections, we compute the variance using
only accepted tables.  In this paper we also investigated
relations with the
exact numbers of tables and $cv^2$ when we have rejections.

In this section, we define the three two-way marginal matrices as following:\\
Suppose we have an observed table $\ve x=(x_{ijk})_{m \times n \times l}$, $i=1,2,\ldots ,m\mbox{, }  
j=1,2,\ldots ,n$, and $k = 1, 2, \ldots, l$;\\
Define:
$si=(X_{+jk})_{n\times l}$, $sj=(X_{i+k})_{m\times l}$, and
$sk=(X_{ij+})_{m\times n}$. 

\begin{ex}[The 3-dimension Semimagic Cube]\label{ex_1}
Suppose $si$, $sj$, and $sk$ are all $3\times 3$ matrices with all 1's inside, that is:
\[
si=sj=sk=
\begin{array}{|c|c|c|}\hline
1 &  1 &  1\\\hline
1 &  1 &  1\\\hline
1 &  1 &  1\\\hline
\end{array}
\]
The real number of tables is $12$. We took $114.7$ seconds to run
$10,000$ samples in the SIS, the estimator is $12$, acceptance rate is $100$\%. Actually, we found that if the acceptance rate is $100$\%, then sample size does not matter in the estimation.
\end{ex}

We used {\tt R} to produce more examples. Examples \eqref{firstex} to \eqref{lastex}
are constructed by the same code but with different values for
parameters. We used the {\tt R} package ``Rlab'' for the following code.
\begin{verbatim}
    seed=6; m=3; n=3; l=4; prob=0.8; N=1000; k=200
    set.seed(seed)
    A=array(rbern(m*n*l,prob),c(m,n,l))
    outinfo=tabinfo(A)
    numtable(N,outinfo,k)
\end{verbatim}
Here prob is the probability of getting $1$ for every Bernoulli
variable, and $N$ is the sample size (the total number of tables
sampled, including both acceptances and rejections).  
Notice that
$cv^2$ is defined as $\frac{Var}{Mean^2}$. 

\begin{ex}[seed=6; m=3; n=3; l=4; prob=0.8]\label{firstex}
Suppose $si$, $sj$, and $sk$ are as following, respectively:
\[                                                                                                           
\begin{array}{|c|c|c|c|}                                                                                       
\hline                                                                                                       
2&2&2&2\\                                                                                                      
\hline                                                                                                       
1&3&2&2\\                                                                                                      
\hline                                                                                                       
2&3&3&2 \\                                                                                                     
\hline                                                                                                       
\end{array}\, ,                                                                                              
\ \ \                                                                                                        
\begin{array}{|c|c|c|c|}                                                                                       
\hline                                                                                                       
2&3&2&2\\                                                                                                      
\hline                                                                                                       
1&3&3&3\\                                                                                                      
\hline                                                                                                       
2&2&2&1 \\                                                                                                     
\hline                                                                                                       
\end{array}\, ,
\ \ \  
\begin{array}{|c|c|c|}                                                                                       
\hline                                                                                                       
3&3&3\\                                                                                                      
\hline                                                                                                       
3&3&4\\                                                                                                      
\hline                                                                                                       
2&2&3 \\                                                                                                     
\hline                                                                                                       
\end{array}\, .                                                                                              
\]
The real number of tables is $3$. An estimator is $3.00762$ with
$cv^2=0.0708$. The whole process took $13.216$ seconds (in {\tt R})
with a $100$\% acceptance rate.
\end{ex}

\begin{ex}[seed=60; m=3; n=4; l=4; prob=0.5]
Suppose $si$, $sj$, and $sk$ are as following, respectively:
\[                                                                                                           
\begin{array}{|c|c|c|c|}                                                                                       
\hline                                                                                                       
2&2&2&1\\                                                                                                      
\hline                                                                                                       
1&1&1&0\\                                                                                                      
\hline                                                                                                       
1&1&1&2 \\                                                                                                     
\hline                                                                                                       
1&1&2&3 \\                                                                                                     
\hline                                                                                                       
\end{array}\, ,                                                                                              
\ \ \                                                                                                        
\begin{array}{|c|c|c|c|}                                                                                       
\hline                                                                                                       
3&3&2&1\\                                                                                                      
\hline                                                                                                       
1&0&2&2\\                                                                                                      
\hline                                                                                                       
1&2&2&3 \\                                                                                                     
\hline                                                                                                       
\end{array}\, ,
\ \ \  
\begin{array}{|c|c|c|c|}                                                                                       
\hline                                                                                                       
3&2&2&2\\                                                                                                      
\hline                                                                                                       
1&0&2&2\\                                                                                                      
\hline                                                                                                       
3&1&1&3 \\                                                                                                     
\hline                                                                                                       
\end{array}\, .                                                                                              
\]
The real number of tables is $5$. An estimator is $4.991026$ with
$cv^2=0.1335$. The whole process took $17.016$ seconds (in {\tt R})
with a $100$\% acceptance rate.
\end{ex}

\begin{ex}[seed=61; m=3; n=4; l=4; prob=0.5]
Suppose $si$, $sj$, and $sk$ are as following, respectively:
\[                                                                                                           
\begin{array}{|c|c|c|c|}                                                                                       
\hline                                                                                                       
1&2&2&1\\                                                                                                      
\hline                                                                                                       
0&1&1&2\\                                                                                                      
\hline                                                                                                       
1&0&2&1 \\                                                                                                     
\hline                                                                                                       
0&1&3&2 \\                                                                                                     
\hline                                                                                                       
\end{array}\, ,                                                                                              
\ \ \                                                                                                        
\begin{array}{|c|c|c|c|}                                                                                       
\hline                                                                                                       
1&2&3&2\\                                                                                                      
\hline                                                                                                       
1&1&2&3\\                                                                                                      
\hline                                                                                                       
0&1&3&1 \\                                                                                                     
\hline                                                                                                       
\end{array}\, ,
\ \ \  
\begin{array}{|c|c|c|c|}                                                                                       
\hline                                                                                                       
3&1&1&3\\                                                                                                      
\hline                                                                                                       
1&2&2&2\\                                                                                                      
\hline                                                                                                       
2&1&1&1 \\                                                                                                     
\hline                                                                                                       
\end{array}\, .                                                                                              
\]
The real number of tables is $8$. An estimator is 8.04964 with
$cv^2=0.2389$. The whole process took $16.446$ seconds (in {\tt R})
with a $100$\% acceptance rate.
\end{ex}

\begin{ex}[seed=240; m=4; n=4; l=4; prob=0.5]
Suppose $si$, $sj$, and $sk$ are as following, respectively:
\[                                                                                                           
\begin{array}{|c|c|c|c|}                                                                                       
\hline                                                                                                       
2&3&3&2\\                                                                                                      
\hline                                                                                                       
1&3&2&1\\                                                                                                      
\hline                                                                                                       
1&2&3&0 \\                                                                                                     
\hline                                                                                                       
4&2&2&2 \\                                                                                                     
\hline                                                                                                       
\end{array}\, ,                                                                                              
\ \ \                                                                                                        
\begin{array}{|c|c|c|c|}                                                                                       
\hline                                                                                                       
2&2&4&1\\                                                                                                      
\hline                                                                                                       
3&2&2&2\\                                                                                                      
\hline                                                                                                       
2&3&3&1 \\                                                                                                     
\hline                                                                                                       
1&3&1&1 \\                                                                                                     
\hline                                                                                                       
\end{array}\, ,
\ \ \  
\begin{array}{|c|c|c|c|}                                                                                       
\hline                                                                                                       
2&2&3&2\\                                                                                                      
\hline                                                                                                       
3&2&1&3\\                                                                                                      
\hline                                                                                                       
3&2&2&2 \\                                                                                                     
\hline                                                                                                       
2&1&0&3 \\                                                                                                     
\hline                                                                                                       
\end{array}\, .                                                                                              
\]
The real number of tables is $8$. An estimator is 8.039938 with
$cv^2=0.2857$. The whole process took $23.612$ seconds (in {\tt R})
with a $100$\% acceptance rate.
\end{ex}

\begin{ex}[seed=1240; m=4; n=4; l=4; prob=0.5]\label{ex6}
Suppose $si$, $sj$, and $sk$ are as following, respectively:
\[                                                                                                           
\begin{array}{|c|c|c|c|}                                                                                       
\hline                                                                                                       
2&3&2&3\\                                                                                                      
\hline                                                                                                       
1&2&3&2\\                                                                                                      
\hline                                                                                                       
2&2&3&2 \\                                                                                                     
\hline                                                                                                       
3&2&3&2 \\                                                                                                     
\hline                                                                                                       
\end{array}\, ,                                                                                              
\ \ \                                                                                                        
\begin{array}{|c|c|c|c|}                                                                                       
\hline                                                                                                       
1&4&1&3\\                                                                                                      
\hline                                                                                                       
4&2&4&2\\                                                                                                      
\hline                                                                                                       
1&2&4&3 \\                                                                                                     
\hline                                                                                                       
2&1&2&1 \\                                                                                                     
\hline                                                                                                       
\end{array}\, ,
\ \ \  
\begin{array}{|c|c|c|c|}                                                                                       
\hline                                                                                                       
2&2&2&3\\                                                                                                      
\hline                                                                                                       
3&3&3&3\\                                                                                                      
\hline                                                                                                       
3&2&2&3 \\                                                                                                     
\hline                                                                                                       
2&1&2&1 \\                                                                                                     
\hline                                                                                                       
\end{array}\, .                                                                                              
\]
The real number of tables is $28$. An  estimator is $26.89940$ with
$cv^2=1.0306$. The whole process took $29.067$ seconds (in {\tt R})
with a $100$\% acceptance rate. It converges even better for sample
size 5000: the estimator becomes $28.0917$, with $cv^2=1.2070$. 
\end{ex}

\begin{ex}[seed=2240; m=4; n=4; l=4; prob=0.5]
Suppose $si$, $sj$, and $sk$ are as following, respectively:
\[                                                                                                           
\begin{array}{|c|c|c|c|}                                                                                       
\hline                                                                                                       
1&2&3&1\\                                                                                                      
\hline                                                                                                       
2&3&2&3\\                                                                                                      
\hline                                                                                                       
2&4&2&1 \\                                                                                                     
\hline                                                                                                       
2&1&4&1 \\                                                                                                     
\hline                                                                                                       
\end{array}\, ,                                                                                              
\ \ \                                                                                                        
\begin{array}{|c|c|c|c|}                                                                                       
\hline                                                                                                       
2&3&2&0\\                                                                                                      
\hline                                                                                                       
3&2&3&2\\                                                                                                      
\hline                                                                                                       
1&3&3&1 \\                                                                                                     
\hline                                                                                                       
1&2&3&3 \\                                                                                                     
\hline                                                                                                       
\end{array}\, ,
\ \ \  
\begin{array}{|c|c|c|c|}                                                                                       
\hline                                                                                                       
2&1&2&2\\                                                                                                      
\hline                                                                                                       
3&2&3&2\\                                                                                                      
\hline                                                                                                       
1&4&2&1 \\                                                                                                     
\hline                                                                                                       
1&3&2&3 \\                                                                                                     
\hline                                                                                                       
\end{array}\, .                                                                                              
\]
The real number of tables is $4$. An estimator is $3.98125$ with
$cv^2=0.0960$. The whole process took $26.96$ seconds (in {\tt R})
with a $100$\% acceptance rate.
\end{ex}

\begin{ex}[seed=3340; m=4; n=4; l=4; prob=0.5]
Suppose $si$, $sj$, and $sk$ are as following, respectively:
\[                                                                                                           
\begin{array}{|c|c|c|c|}                                                                                       
\hline                                                                                                       
2&4&1&3\\                                                                                                      
\hline                                                                                                       
1&2&1&2\\                                                                                                      
\hline                                                                                                       
1&1&0&3 \\                                                                                                     
\hline                                                                                                       
4&1&0&2 \\                                                                                                     
\hline                                                                                                       
\end{array}\, ,                                                                                              
\ \ \                                                                                                        
\begin{array}{|c|c|c|c|}                                                                                       
\hline                                                                                                       
2&1&1&2\\                                                                                                      
\hline                                                                                                       
3&1&1&3\\                                                                                                      
\hline                                                                                                       
1&2&0&2 \\                                                                                                     
\hline                                                                                                       
2&4&0&3 \\                                                                                                     
\hline                                                                                                       
\end{array}\, ,
\ \ \  
\begin{array}{|c|c|c|c|}                                                                                       
\hline                                                                                                       
3&1&1&1\\                                                                                                      
\hline                                                                                                       
3&1&2&2\\                                                                                                      
\hline                                                                                                       
1&2&1&1 \\                                                                                                     
\hline                                                                                                       
3&2&1&3 \\                                                                                                     
\hline                                                                                                       
\end{array}\, .                                                                                              
\]
The real number of tables is $2$. An estimator is $2$ with
$cv^2=0$. The whole process took $15.214$ seconds (in {\tt R}) with a
$100\%$ acceptance rate. 
\end{ex}

\begin{ex}[seed=3440; m=4; n=4; l=4; prob=0.5]\label{eg5_9}
Suppose $si$, $sj$, and $sk$ are as following, respectively:
\[                                                                                                           
\begin{array}{|c|c|c|c|}                                                                                       
\hline                                                                                                       
1&3&1&3\\                                                                                                      
\hline                                                                                                       
1&1&2&2\\                                                                                                      
\hline                                                                                                       
2&3&1&0 \\                                                                                                     
\hline                                                                                                       
3&2&2&3 \\                                                                                                     
\hline                                                                                                       
\end{array}\, ,                                                                                              
\ \ \                                                                                                        
\begin{array}{|c|c|c|c|}                                                                                       
\hline                                                                                                       
2&2&2&2\\                                                                                                      
\hline                                                                                                       
2&1&2&1\\                                                                                                      
\hline                                                                                                       
1&3&1&2 \\                                                                                                     
\hline                                                                                                       
2&3&1&3 \\                                                                                                     
\hline                                                                                                       
\end{array}\, ,
\ \ \  
\begin{array}{|c|c|c|c|}                                                                                       
\hline                                                                                                       
3&1&1&3\\                                                                                                      
\hline                                                                                                       
1&2&1&2\\                                                                                                      
\hline                                                                                                       
2&0&3&2 \\                                                                                                     
\hline                                                                                                       
2&3&1&3 \\                                                                                                     
\hline                                                                                                       
\end{array}\, .                                                                                              
\]
The real number of tables is $12$. An estimator is $12.04838$ with
$cv^2=0.7819733$. The whole process took $27.074$ seconds (in {\tt R})
with a $85.9$\% acceptance rate.
\end{ex}

\begin{ex}[seed=5440; m=4; n=4; l=4; prob=0.5]\label{eg5_10}
Suppose $si$, $sj$, and $sk$ are as following, respectively:
\[                                                                                                           
\begin{array}{|c|c|c|c|}                                                                                       
\hline                                                                                                       
2&1&0&1\\                                                                                                      
\hline                                                                                                       
2&3&1&2\\                                                                                                      
\hline                                                                                                       
3&1&2&1 \\                                                                                                     
\hline                                                                                                       
1&3&2&2 \\                                                                                                     
\hline                                                                                                       
\end{array}\, ,                                                                                              
\ \ \                                                                                                        
\begin{array}{|c|c|c|c|}                                                                                       
\hline                                                                                                       
2&3&2&1\\                                                                                                      
\hline                                                                                                       
2&1&2&3\\                                                                                                      
\hline                                                                                                       
2&1&0&1 \\                                                                                                     
\hline                                                                                                       
2&3&1&1 \\                                                                                                     
\hline                                                                                                       
\end{array}\, ,
\ \ \  
\begin{array}{|c|c|c|c|}                                                                                       
\hline                                                                                                       
1&2&2&3\\                                                                                                      
\hline                                                                                                       
1&1&3&3\\                                                                                                      
\hline                                                                                                       
1&3&0&0 \\                                                                                                     
\hline                                                                                                       
1&2&2&2 \\                                                                                                     
\hline                                                                                                       
\end{array}\, .                                                                                              
\]
The real number of tables is $9$. An estimator is $8.882672$ with
$cv^2=0.7701368$. The whole process took $30.171$ seconds (in {\tt R})
with a $100$\% acceptance rate. Another result for the same sample
size is: an estimator is $8.521734$, $cv^2=0.6695902$. You can find
that the latter has a slightly better $cv^2$ but a slightly worse
estimator. We'll discuss more in  Section \ref{dis}. 
\end{ex}

\begin{ex}[seed=122; m=4; n=4; l=5; prob=0.2]
Suppose $si$, $sj$, and $sk$ are as following, respectively:
\[                                                                                                           
\begin{array}{|c|c|c|c|c|}                                                                                       
\hline                                                                                                       
2&0&3&3&2\\                                                                                                      
\hline                                                                                                       
0&0&1&0&0\\                                                                                                      
\hline                                                                                                       
1&0&1&1&1 \\                                                                                                     
\hline                                                                                                       
0&1&0&1&0 \\                                                                                                     
\hline                                                                                                       
\end{array}\, ,                                                                                              
\ \ \                                                                                                        
\begin{array}{|c|c|c|c|c|}                                                                                       
\hline                                                                                                       
1&0&0&2&1\\                                                                                                      
\hline                                                                                                       
1&0&2&1&1\\                                                                                                      
\hline                                                                                                       
1&1&1&1&1 \\                                                                                                     
\hline                                                                                                       
0&0&2&1&0 \\                                                                                                     
\hline                                                                                                       
\end{array}\, ,
\ \ \  
\begin{array}{|c|c|c|c|}                                                                                       
\hline                                                                                                       
3&0&0&1\\                                                                                                      
\hline                                                                                                       
4&1&0&0\\                                                                                                      
\hline                                                                                                       
1&0&3&1 \\                                                                                                     
\hline                                                                                                       
2&0&1&0 \\                                                                                                     
\hline                                                                                                       
\end{array}\, .                                                                                              
\]
The real number of tables is $5$. An estimator is $4.93625$ with
$cv^2=0.2035$. The whole process took $21.325$ seconds (in {\tt R})
with a $100$\% acceptance rate.
\end{ex}

\begin{ex}[seed=222; m=4; n=4; l=5; prob=0.2]
Suppose $si$, $sj$, and $sk$ are as following, respectively:
\[                                                                                                           
\begin{array}{|c|c|c|c|c|}                                                                                       
\hline                                                                                                       
1&0&1&1&1\\                                                                                                      
\hline                                                                                                       
2&1&0&1&2\\                                                                                                      
\hline                                                                                                       
0&1&1&1&0 \\                                                                                                     
\hline                                                                                                       
1&1&1&1&1 \\                                                                                                     
\hline                                                                                                       
\end{array}\, ,                                                                                              
\ \ \                                                                                                        
\begin{array}{|c|c|c|c|c|}                                                                                       
\hline                                                                                                       
2&1&0&0&2\\                                                                                                      
\hline                                                                                                       
1&2&1&2&1\\                                                                                                      
\hline                                                                                                       
1&0&1&1&1 \\                                                                                                     
\hline                                                                                                       
0&0&1&1&0 \\                                                                                                     
\hline                                                                                                       
\end{array}\, ,
\ \ \  
\begin{array}{|c|c|c|c|}                                                                                       
\hline                                                                                                       
2&3&0&0\\                                                                                                      
\hline                                                                                                       
1&3&2&1\\                                                                                                      
\hline                                                                                                       
0&0&1&3 \\                                                                                                     
\hline                                                                                                       
1&0&0&1 \\                                                                                                     
\hline                                                                                                       
\end{array}\, .                                                                                              
\]
The real number of tables is $2$. An estimator is $2$ with
$cv^2=0$. The whole process took $19.064$ seconds (in {\tt R}) with a
$100$\% acceptance rate. 
\end{ex}

\begin{ex}[seed=322; m=4; n=4; l=5; prob=0.2]\label{lastex}\label{eg5_13}
Suppose $si$, $sj$, and $sk$ are as following, respectively:
\[                                                                                                           
\begin{array}{|c|c|c|c|c|}                                                                                       
\hline                                                                                                       
1&1&1&1&1\\                                                                                                      
\hline                                                                                                       
1&1&1&1&1\\                                                                                                      
\hline                                                                                                       
1&2&0&0&1 \\                                                                                                     
\hline                                                                                                       
2&0&1&1&2 \\                                                                                                     
\hline                                                                                                       
\end{array}\, ,                                                                                              
\ \ \                                                                                                        
\begin{array}{|c|c|c|c|c|}                                                                                       
\hline                                                                                                       
0&0&1&1&0\\                                                                                                      
\hline                                                                                                       
1&0&1&0&1\\                                                                                                      
\hline                                                                                                       
2&2&0&1&2 \\                                                                                                     
\hline                                                                                                       
2&2&1&1&2 \\                                                                                                     
\hline                                                                                                       
\end{array}\, ,
\ \ \  
\begin{array}{|c|c|c|c|}                                                                                       
\hline                                                                                                       
0&2&0&0\\                                                                                                      
\hline                                                                                                       
1&0&0&2\\                                                                                                      
\hline                                                                                                       
1&3&1&2 \\                                                                                                     
\hline                                                                                                       
3&0&3&2 \\                                                                                                     
\hline                                                                                                       
\end{array}\, .                                                                                              
\]
The real number of tables is $5$. An estimator is $4.992$ with
$cv^2=0.2179682$. The whole process took $23.25$ seconds (in {\tt R})
with a $85.2$\% acceptance rate.
\end{ex}

\begin{summ}[Summarize the results from Example \eqref{firstex} to
  Example \eqref{lastex}]
This is only a summary of main results of those examples in Table
\ref{tab_res}. For all results 
here, we set the sample size $1,000$. We will discuss these results in Section
\ref{dis}. 

\begin{table}[!htp]
\begin{center}
\begin{tabular}{ccrrrr}
\toprule[1.2pt] %
Dimension & Example & \# tables & Estimation & $cv^2$ & Acceptance rate \\\midrule[1.2pt]
$3\times 3\times 4$ & 5.2 & 3 & 3.00762 & 0.0708 & 100\% \\\midrule
$3\times 4\times 4$ & 5.3 & 5 & 4.991026 & 0.1335 & 100\% \\
& 5.4 & 8 & 8.04964 &0.2389 & 100\% \\\midrule
$4\times 4\times 4$ & 5.5 & 8 & 8.039938 & 0.2857 & 100\% \\
& 5.6 & 28 & 26.89940 & 1.0306 & 100\% \\
& 5.7 & 4 & 3.98125 & 0.0960 & 100\% \\
& 5.8 & 2 & 2 & 0 & 100\% \\
& 5.9 & 12 & 12.04838 & 0.7820 & 85.9\% \\
& 5.10 & 9 & 8.882672 & 0.7701 & 100\% \\\midrule
$4\times 4\times 5$ & 5.11 & 5 & 4.93625 & 0.2035& 100\% \\
& 5.12 & 2 & 2 & 0 & 100\% \\
& 5.13 & 5 & 4.992 & 0.2180 & 85.2\% \\\bottomrule[1.2pt]
\end{tabular}
\vskip 0.1in
\caption{Summary of Examples \eqref{firstex} - \eqref{lastex}}\label{tab_res}
\end{center}
\end{table}

\end{summ}

\begin{ex}[High-dimension Semimagic Cubes]\label{lastex2}\label{eg5_15}
In this example, we consider $m \times n\times l$ tables for $m=n=l = 4,
\ldots, 10$ such that each marginal sum equals to $1$. The results are
summarized in Table \ref{tab_res2}.

\begin{table}[!htp]
\begin{center}
\begin{tabular}{ccrrrr}
\toprule[1.2pt] %
Dimension $m$ & $N$ &CPU time (sec)& Estimation & $cv^2$ & Acceptance rate \\\hline
$4$ & $1000$ & $32.44$ & 568.944 & 0.26 & $100\%$\\
\     & $10000$ & $324.18$ & 571.1472 &0.27 & $100\%$\\\hline
$5$ & $1000$ & $60.39$ & 161603.5 & 0.18 & $99\%$\\
\     & $10000$ & $605.45$ & 161439.3 &0.18 & $99.2\%$\\\hline

$6$ & $1000$ & $102.66$ & 801634023 & 0.58 & $98.3\%$\\
\     & $10000$ & $1038.46$ & 819177227 &0.45 & $98.8\%$\\\hline

$7$ & $1000$ & $158.55$ & 6.08928e+13 & 0.60 & $97\%$\\
\     & $10000$ & $1590.84$ & 6.146227e+13 & 0.64 & $97.7\%$\\\hline

$8$ & $1000$ & $234.53$ & 1.080208e+20 & 1.07 & $95.6\%$\\
\     & $10000$ & $2300.91$ &1.099627e+20  & 1.00 & $96.5\%$\\\hline

$9$ & $1000$ & $329.17$ & 5.845308e+27 & 1.46 & $94\%$\\
\     & $10000$ & $3238.1$ & 5.684428e+27 &1.59& $95.3\%$ \\\hline

$10$ & $1000$ & $451.24$ & 9.648942e+36 & 1.44 & $93.3\%$\\
\       & $10000$ & $4425.12$&  9.73486e+36& 1.73& $93.3\%$\\\bottomrule[1.2pt]
\end{tabular}
\vskip 0.1in
\caption{Summary of computational results on $m \times n\times l$ tables for $m=n=l = 4,
\ldots, 10$. 
  The all marginal sums are equal to one in this example.}\label{tab_res2}
\end{center}
\end{table}

\end{ex}

\begin{ex}[High-dimension Semimagic Cubes continues]\label{lastex3}\label{eg5_16}
In this example, we consider $m \times n\times l$ tables for $m=n=l = 4,
\ldots, 10$ such that each marginal sum equals to $s$. The results are
summarized in Table \ref{tab_res3}. In this example, we set the sample size $N = 1000$.

\begin{table}[!htp]
\begin{center}
\begin{tabular}{ccrrrr}
\toprule[1.2pt] %
Dimension $m$ & $s$ &CPU time (sec)& Estimation & $cv^2$ & Acceptance rate \\\hline
$4$ & $2$ & $27.1$ & 51810.36 & 0.66 & $97.7\%$\\\hline
$5$ & $2$ & $58.1$ & 25196288574 & 1.69 & $97.5\%$\\\hline
$6$ & $2$ & $97.1$ & 6.339628e+18 & 2.56 & $94.8\%$\\
\     & $3$ & $99.3$ & 1.269398e+22&2.83 & $96.5\%$\\\hline

$7$ & $2$ & $150.85$ & 1.437412e+30 & 4.76 & $93.1\%$\\
\     & $3$ & $166.68$ & 2.365389e+38 & 25.33 & $96.7\%$\\\hline

$8$ & $2$ & $ 229.85$ & 5.369437e+44 & 6.68 & $89.8\%$\\
\     & $3$ & $ 256.70$ & 3.236556e+59 & 7.05 & $94.5\%$\\
\     & $4$ & $328.52$ &2.448923e+64  & 11.98 & $94.3\%$\\\hline

$9$ & $2$ & $319.32$ & 4.416787e+62 & 8.93 & $85.7\%$\\
\     & $3$ & $376.67$ & 7.871387e+85& 15.23 & $91.6\%$\\
\     & $4$ & $549.73$ & 2.422237e+97 &14.00 & $93.4\%$\\\hline

$10$ & $2$ & $429.19$ & 2.166449e+84 & 10.46 & $83.3\%$\\
\       & $3$ & $527.14$ & 6.861123e+117 &26.62& $90\%$ \\

\       & $4$ & $883.34$ & 3.652694e+137& 33.33 & $93.8\%$\\
\       & $5$ & $1439.50$&  1.315069e+144& 46.2& $91.3\%$\\\bottomrule[1.2pt]
\end{tabular}
\vskip 0.1in
\caption{Summary of computational results on $m \times n\times l$ tables for $m=n=l = 4,
\ldots, 10$. 
  The all marginal sums are equal to $s$ in this example.  The sample
  $N = 1000$ in this example. }\label{tab_res3}
\end{center}
\end{table}


\end{ex}

\begin{ex}[Bootstrap-t confidence interval of Semimagic Cubes]\label{lastex4}\label{eg5_17}
As we can see that in Table \ref{tab_res3}, generally we have larger
$cv^2$ when the number of tables is larger, and in this case, the
estimator we get via the SIS procedure might vary greatly in different
iterations. Therefore, we might want to compute a $(1-\alpha)100\%$ confidence interval
for each estimator via a non-parametric bootstrap method (see
Appendix \ref{nonpara} for a pseudo code for a non-parametric
bootstrap method to get the $(1-\alpha)100\%$ confidence interval for
$|\Sigma|$).  
See Table \ref{tab_res5} for some results of Bootstrap-t $95\%$
confidence intervals ($\alpha = 0.05$).

\begin{table}[!htp]
\begin{center} {\footnotesize
\begin{tabular}{|c|c|rrr|rrr|r|}
\toprule[1.2pt]
 & & \multicolumn{3}{c|}{Estimation} & \multicolumn{3}{c|}{$cv^2$} & \\\cline{3-8}

Dim & s  & \multicolumn{1}{c|}{$\widehat{|\Sigma|}$} & \multicolumn{1}{c}{Lower $95\%$} &
\multicolumn{1}{c|}{Upper $95\%$} & \multicolumn{1}{c|}{$\widehat{cv^2}$} & \multicolumn{1}{c}{Lower $95\%$} & \multicolumn{1}{c|}{Upper $95\%$} & \multicolumn{1}{c|}{Acceptance Rate} \\\midrule[1.2pt]

7 & 2 & 1.306480e+30 & 1.156686e+30 & 1.468754e+30 & 3.442306 & 2.678507 & 4.199513 & $93.3\%$ \\
   & 3 & 3.033551e+38 & 2.245910e+38 & 4.087225e+38 & 22.84399 & 8.651207 & 35.080408 & $96.2\%$ \\
\hline

8 & 2 & 5.010225e+44 & 4.200752e+44 & 5.902405e+44 & 6.712335 & 4.539368 & 8.590578 & $90.4\%$ \\
   & 3 & 2.902294e+59 & 2.389625e+59 & 3.484405e+59 & 9.047914 & 5.680128 & 12.797488 & $93.1\%$ \\
   & 4 & 2.474874e+64 & 1.847911e+64 & 3.295986e+64 & 21.53559 & 5.384647 & 32.166086 & $94.6\%$ \\
\hline

9 & 2 & 4.548401e+62 & 3.682882e+62 & 5.593370e+62 & 10.07973 & 4.886817 & 15.406899 & $87.1\%$ \\
   & 3 & 9.702672e+85 & 7.189849e+85 & 1.250875e+86 & 18.65302 & 11.33462 & 23.77980 & $92.5\%$ \\
   & 4 & 2.023034e+97 & 1.547951e+97 & 2.561084e+97 & 14.96126 & 10.20331 & 19.09515 & $92.2\%$ \\
\hline

10 & 2 & 2.570344e+84 & 1.908609e+84 & 3.339243e+84 & 17.83684 & 9.785778 & 24.231544 & $84.8\%$ \\
   & 3 & 8.68783e+117 & 5.92233e+117 & 1.22271e+118 & 29.67200 & 18.64549 & 37.64892 & $90.2\%$ \\
   & 4 & 4.12634e+137 & 2.94789e+137 & 5.52727e+137 & 23.36831 & 15.32719 & 31.02614 & $92\%$ \\
   & 5 & 1.54956e+144 & 9.85557e+143 & 2.24043e+144 & 39.06521 &  20.23674 & 53.60838 & $91.8\%$ 
\\\bottomrule[1.2pt]

\end{tabular} }
\end{center}
\vskip 0.1in
\caption{Summary of confidence intervals. Dimensions and marginals$=s$
  are defined same with Table \ref{tab_res3}.  $\widehat{|\Sigma|}$
  means an estimator of ${|\Sigma|}$ and $\widehat{cv^2}$ is an
  estimator of $cv^2$. The sample size for the SIS procedure
  is $N = 1000$ and the sample size for bootstraping is $B=5000$. Only cases with
  relatively large $cv^2$ are involved.} 
\label{tab_res5}
\end{table}

\end{ex}

\section{Experiment with Sampson's data set}\label{sam}

Sampson recorded the social interactions among a group of monks
when he was visiting there as an experimenter on vision. He collected numerous
sociometric rankings \cite{Breiger,sampson}.  The data is organized
as a $18 \times 18\times 
10$ table and one can find the full data sets at
\url{http://vlado.fmf.uni-lj.si/pub/networks/data/ucinet/UciData.htm#sampson}. 
Each layer of $18 \times 18$ table represents a social relation
between 18 monks at some time point.  
Most of the present data are retrospective, collected after the
breakup occurred. They concern a period during which a new cohort
entered the monastery near the end of the study but before the major
conflict began. The exceptions are ``liking'' data gathered at three
times: SAMPLK1 to SAMPLK3 - that reflect changes in group sentiment
over time (SAMPLK3 was collected in the same wave as the data
described below). 
In the data set four relations are coded, with separate matrices for positive and
negative ties on the 10 relation: esteem (SAMPES) and
disesteem (SAMPDES); liking (SAMPLK which are SAMPLK1 to SAMPLK3) and
disliking (SAMPDLK); positive 
influence (SAMPIN) and negative influence (SAMPNIN); praise (SAMPPR)
and blame (SAMPNPR).  
In the original data set they listed top three choices and recorded as
ranks.  However, we set these ranks as an indicator (i.e., if they are
in the top three choices, then we set one and else, zero).  

We ran the SIS procedure with $N = 100000$ and a bootstrap sample size
$B = 50000$.  An estimator was  1.704774e+117 with its $95\%$
confidence interval, [1.119321e+117 2.681264e+119]
and $cv^2 = 621.4$ with its $95\%$
confidence interval,  [324.29, 2959.65].  The CPU time was $70442$ seconds.  The acceptance
rate is 3\%.

\section{Discussion}\label{dis}

In this paper we do not have a
sufficient and necessary condition for the existence of the three-way
zero-one table so we cannot avoid rejection. However, since the SIS
procedure gives an unbiased estimator,  we may only need a 
small sample size as long as it converges. For example, in Table
\ref{tab_res},  all estimators with 
sample size $1000$ are exactly the same as the true numbers of tables because they
all converge very quickly. Also note 
that an acceptance rate does not depend on a sample size.  Thus, it would be
interesting to investigate the convergence rate of the SIS procedure with
CP for zero-one three-way tables.

It seems that the convergence rate is slower when we have a ``large''
table (here ``large'' means in terms of $|\Sigma|$ rather than
its dimension, i.e., the number of cells).  A large estimator 
$\widehat{|\Sigma|}$ usually corresponds to a larger $cv^2$, and this
often comes with large  
variations of $\widehat{|\Sigma|}$ and $cv^2$.   This means that if we
have a large $|\Sigma|$, more likely we get extremely  
larger $\widehat{|\Sigma|}$ and $cv^2$ and different iterations can give very different 
results. For example, we ran three iterations for the $8\times 8\times 8$
semimagic cube with all marginals  
equal to $3$ and we got the following results: estimator =3.236556e+59 with
$cv^2=7.049114$; estimator =2.902294e+59
with $cv^2=9.047914$; and estimator =3.880133e+59 with
$cv^2=55.59179$. Fortunately, even though we have a  
large $|\Sigma|$, our acceptance rate is still high and a
computational time seems to still be  
attractive.  Thus, when one finds a large estimation or a large $cv^2$,
we recommend to
apply several iterations and pick the result with the smallest $cv^2$.  We should 
always compare $cv^2$ in a large scale.  However, a small improvement does
not necessarily mean a  
better estimator  (see Example \ref{eg5_10}). 

For calculating the bootstrap-t confidence 
intervals,  we often have a larger confidence interval when we have a larger 
$cv^2$, and this confidence interval might be less informative and less
reliable. Therefore we suggest to use the result with the smallest
$cv^2$ for bootstraping  
procedure. In Table \ref{tab_res5} we showed only confidence intervals 
for 
semimagic cubes with $m=n = l =  7, \ldots ,10$ in Example
\ref{eg5_17} because of the following reason: When
$cv^2$ is very small, computing bootstrap-t confidence interval  
does not make much sense, since the estimation has already converged.

For an experiment with Sampson's data set, we have observed a very low
acceptance rate compared with experimental studies on simulated data
sets.  We are investigating why this happens and how to increase the
acceptance rates.

In \cite{chen2005}, the Gale--Ryser Theorem was used to obtain an SIS
procedure without rejection for two-way zero-one tables.  However,
for three-way table cases, it seems very difficult because we
naturally have structural zeros and trivial cases on a process of
sampling one table.  In \cite{chen2007} Chen showed a version of
Gale--Ryser Theorem for structural zero for two-way zero-one tables, but
it assumes that there is at most one structural zero in each row and
column.  In general there are usually more than one in each row and column.

In this paper the target distribution is the uniform distribution.  We
are sampling a table from the set of all zero-one tables satisfying
the given marginals as close as uniformly via the SIS procedure with CP.
For a goodness-of-fit test one might want to sample a table from the
set of all zero-one tables satisfying the given marginals with the
hypergeometric distribution.  We are currently working on how to
sample a table via the SIS procedure with CP for the hypergeometric
distribution.  



\section{Acknowledgement}
The authors would like to thank Drs.~Stephen Fienberg and Yuguo Chen for useful
conversations.  

\bibliographystyle{plain}
\bibliography{sis}

\appendix
\section{Non-parametric bootstrap method}\label{nonpara}

In this section we explain how to use a non-parametric
bootstrap method to get the $(1-\alpha)100\%$ confidence interval for
$|\Sigma|$. Notice that the bootstrap sample size is fixed as B, and
notations here are consistent with Section \ref{sec:sis}. 

\begin{itemize}
\item[\textbf{(1)}] \textbf{Drawing pseudo dataset.}
\begin{enumerate}
\item[{\bf Concept}]
In an SIS procedure with sample size N,  we get a sequence of random
tables ${\bf X_1}, \ldots , {\bf X_N}$. Define ${\bf
  Y_i}=\frac{\mathbb{I}_{{\bf X_i} \in \Sigma}}{q({\bf X_i})},\
i=1,\ldots,N$ where $q({\bf X})$ is the trial distribution, then ${\bf
  Y_1}, \ldots , {\bf Y_N}$ is a sequence of i.i.d random
variables. This means that it makes sense to consider the empirical
distribution of $\bf Y_i$, which is nonparametric maximum likelihood
estimator of the real distribution of $\bf Y_i$ (actually, as $\bf
Y_i$ can only take finitely many values, the empirical distribution
becomes the maximum likelihood estimator of the real
distribution). Draw a pseudo sample ${\bf Y^*_1}, \ldots , {\bf
  Y^*_N}$ from the empirical distribution. 
\item[{\bf Algorithm}]
Use the SIS procedure to get ${\bf Y_i}=\frac{\mathbb{I}_{{\bf X_i} \in
    \Sigma}}{q({\bf X_i})},\ i=1,\ldots,N$, which should be just a
sequence of numbers. Draw N elements from this sequence with
replacement. 
\end{enumerate}

\item[\textbf{(2)}] \textbf{One Bootstrap replication.}
\begin{enumerate}
\item[{\bf Concept}]
Consider the pseudo sample ${\bf Y^*_1}, \ldots , {\bf Y^*_N}$ as a
"new" sample from the empirical distribution, then the cumulative
distribution function (CDF) of 
$\widehat{\theta}^*=T({\bf Y^*_1}, \ldots , {\bf Y^*_N})$ is a consistent
estimator of the CDF of $\widehat{\theta}=T({\bf Y_1}, \ldots , {\bf
  Y_N})$. Here we can consider our estimator of $|\Sigma|$: 
\[
\widehat{|\Sigma|}=\widehat{\theta_1}=T_1({\bf Y_1}, \ldots , {\bf Y_N})=\frac{1}{N} \sum_{i = 1}^N {\bf Y_i}
\]
And the $cv^2$:
\[
\widehat{cv^2}=\widehat{\theta_2}=T_2({\bf Y_1}, \ldots , {\bf Y_N})=\frac{\sum_{i = 1}^N \{\bf Y_i - \left[\sum_{j =1}^N {\bf Y_j} \right] / N  \}^2 /(N - 1)}{\{ \left[ \sum_{j = 1}^N \bf Y_j \right] / N\}^2 }
\]
\item[{\bf Algorithm}]
Treat the pseudo sample as a sample from the SIS and compute the statistics based on it. That means, this bootstrap replication can be got by:
\[
\widehat{|\Sigma|}^{*1} =\frac{1}{N} \sum_{i = 1}^N {\bf Y^*_i};\ \widehat{cv^2}_{*1} =cv^2 of ({\bf Y^*_1}, \ldots , {\bf Y^*_N})
\]
\end{enumerate}

\item[\textbf{(3)}] \textbf{Bootstrap-t Confidence Interval.}
\begin{enumerate}
\item[{\bf Concept}]
Repeat the previous two steps until we get B Bootstrap replications: $\widehat{\theta_i}^{*1},\ldots,\widehat{\theta_i}^{*B},\ i=1,2$. The empirical distribution of $\widehat{\theta_i}^*$ is the nonparametric maximum likelihood estimator of CDF of $\widehat{\theta_i}^*$, and the latter is consistent estimator of the CDF of $\widehat{\theta_i}$. So we can use $(\frac {\alpha}{2})100_{th}$ and $(1-\frac {\alpha}{2})100_{th}$ percentiles of the empirical distribution as our confidence Interval.
\item[{\bf Algorithm}]
Repeat the previous two steps for B times. For $\{ \widehat{|\Sigma|}^{*1}, \ldots, \widehat{|\Sigma|}^{*B} \}$, define $\widehat{|\Sigma|}_{(a)}^*$ as the $100a_{th}$ percentile of the list of values. Then bootstrap-t $(1-\alpha)100\%$ confidence interval of $\widehat{|\Sigma|}$ is $[\widehat{|\Sigma|}_{(\alpha/2)}^*,\widehat{|\Sigma|}_{(1-\alpha/2)}^*]$. Similarly we can get confidence interval for $\widehat{cv^2}$.
\end{enumerate}

\end{itemize}

\end{document}